\documentclass[11pt]{amsart}
\usepackage{amsmath,amsfonts,amssymb,amsthm}
\usepackage[alphabetic]{amsrefs}
\usepackage{hyperref}
\usepackage{graphicx,color}
\usepackage{enumerate}

\newtheorem{theorem}{Theorem}[section]
\newtheorem{lemma}[theorem]{Lemma}
\newtheorem{prop}[theorem]{Proposition}
\newtheorem{cor}[theorem]{Corollary}

\newtheorem*{mainmain}{Main Theorem}
\newtheorem{main}{Theorem}

\theoremstyle{definition}
\newtheorem{definition}[theorem]{Definition}
\newtheorem{rem}[theorem]{Remark}
\newtheorem{exa}[theorem]{Example}

\newtheorem*{question}{Questions}

\newcommand{\Z}{\mathbb{Z}}
\newcommand{\R}{\mathbb{R}}
\newcommand{\kkk}{\mathbf{k}}
\newcommand{\lk}{\mathrm{lk}}

\newcommand{\ics}{invariant cocompact subcomplex}
\newcommand{\sics}{c.s.}
\newcommand{\gal}{$\mathrm{gal}_Z(T)$}

\newcommand{\kol}{\color{black}}

\begin{document}

\title[Tits Alternative for $2$-dimensional recurrent complexes]{Tits Alternative for groups acting properly on $2$-dimensional recurrent complexes
}

\author[D.~Osajda]{Damian Osajda$^{\dag}$}
\address{Instytut Matematyczny,
	Uniwersytet Wroc\l awski\\
	pl.\ Grun\-wal\-dzki 2/4,
	50--384 Wroc\-{\l}aw, Poland}
\address{Institute of Mathematics, Polish Academy of Sciences\\
	\'Sniadeckich 8, 00-656 War\-sza\-wa, Poland}
\email{dosaj@math.uni.wroc.pl}
\thanks{$\dag \ddag$ Partially supported by (Polish) Narodowe Centrum Nauki, UMO-2018/30/M/ST1/00668. This work was partially supported by the grant 346300 for IMPAN from the Simons Foundation and the matching 2015-2019 Polish MNiSW fund.}

\author[P.~Przytycki]{Piotr Przytycki$^{\ddag}$\\
	\\	{\tiny{{with an appendix by J.\ McCammond, D.\ Osajda, and P.\ Przytycki}}}}
\address{
Department of Mathematics and Statistics,
McGill University,
Burnside Hall,
805 Sherbrooke Street West,
Montreal, QC,
H3A 0B9, Canada}

\email{piotr.przytycki@mcgill.ca}

\thanks{$\ddag$ Partially supported by NSERC and FRQNT}

\address{
	Mathematics Department,
	UC Santa Barbara,
	Santa Barbara, CA 93106}
\email{jon.mccammond@math.ucsb.edu}

\begin{abstract}
We prove the Tits Alternative for groups acting on $2$-dimen\-sional ``recurrent'' complexes {\kol with uniformly bounded cell stabilisers}. This class
of complexes includes, among others: $2$-dimensional {\kol Euclidean} buildings, $2$-di\-men\-sio\-nal systolic complexes,
$B(6)$-small cancellation complexes, and standard Cayley complexes for Artin groups of extra-large type.

In the appendix written jointly with Jon McCammond we extend the result to a class of $2$-dimensional
Artin groups containing all large-type Artin groups.
\end{abstract}

\maketitle

\section{Introduction}

Tits proved that every finitely generated linear group is either virtually solvable or contains a nonabelian free group \cite{Tits1972}.
In other words, each linear group $\mathbf{GL}_n(\kkk)$ satisfies the \emph{Tits Alternative}, saying that each of its finitely generated subgroups is virtually solvable or contains a nonabelian free group. It is believed that the Tits Alternative is common among `nonpositively curved' groups. However, up to now it has been shown only for few particular classes of groups.
Most notably, for: Gromov-hyperbolic groups \cite{Gromov1987}, mapping class groups \cite{Ivanov1984,McCarthy1985}, $\mathrm{Out}(F_n)$ \cite{BFH2000, BFH2}, fundamental groups of closed $3$-manifolds (by geometrisation, cf.\ \cite{KoZa2007}), fundamental groups of some nonpositively curved real-analytic $4$-manifolds \cite{Xie2004}, $\mathrm{CAT}(0)$ cubical groups \cite{SagWis2005}.
Whether $\mathrm{CAT}(0)$ groups satisfy the Tits Alternative remains an open question, even in the case
of groups acting properly and cocompactly on $2$-dimensional $\mathrm{CAT}(0)$ complexes.
\medskip

In this article we prove the Tits Alternative for groups acting 
on triangle complexes that are ``recurrent''. Here a \emph{triangle complex} is a $2$-dimensional simplicial complex~$X$ built of geodesic Euclidean triangles, see \cite[I.7.2]{BriHaf1999}. We postpone the general definition of ``recurrent'' till Section~\ref{sec:like}, and here we discuss examples instead. All the group actions that we consider are by combinatorial isometries. An action of a group $G$ is \emph{without inversions} if each element of~$G$ stabilising a cell fixes it pointwise.
{\kol The action is \emph{almost free} if there is a bound on the order of cell stabilisers. Note that an almost free action on a triangle complex with finitely many isometry types of simplices is proper in the sense of \cite[I.8.2]{BriHaf1999}.}

\begin{mainmain}
\label{t:tfTits}
Let $X$ be a simply connected triangle complex that is recurrent w.r.t.\ a finitely generated group $G$ acting {\kol almost freely} and without inversions. 
Then $G$ is virtually cyclic, or virtually $\mathbb{Z}^2$, or contains a nonabelian free group.
\end{mainmain}

In particular, by Remark~\ref{rem:hereditary} the same conclusion will hold for any finitely generated subgroup of $G$. In other words, $G$ satisfies the {Tits Alternative}.

For example, let $X$ be a $2$-dimensional {\kol Euclidean} building or a $2$-dimensio\-nal \emph{systolic complex}, which is a $\mathrm{CAT}(0)$ triangle complex with all edges of length $1$ and all triangles equilateral. We will show in Corollaries~\ref{cor:sys} and~\ref{cor:build} that $X$ has a subdivision recurrent with respect to any automorphism group of $X$. This implies the following for finitely generated subgroups of $G$.

\begin{main}
\label{main:cat0}
Let $X$ be a $2$-dimensional {\kol Euclidean} building or a $2$-dimensio\-nal systolic complex. Suppose that $G$ acts {\kol almost freely} on~$X$ (e.g.\ $G$ acts on $X$ properly and cocompactly). Then any subgroup of $G$ is virtually cyclic, or virtually $\mathbb{Z}^2$, or contains a nonabelian free group.
\end{main}

Ballmann and Brin proved that if $X$ is any $2$-dimensional $\mathrm{CAT}(0)$ complex, and $G$ acts on $X$ properly and cocompactly, then $G$ itself is virtually cyclic, or virtually $\mathbb{Z}^2$, or contains a nonabelian free group \cite{BallBrin1995}. However, it was only very recently that we were able to perform with Norin a first step to understand the subgroups of $G$, by proving that each of them is finite or contains $\Z$ \cite{NOP}.

Note that in Theorem~\ref{main:cat0}, as in many other applications of the Main Theorem, we will be able to remove the assumption that the group is finitely generated; {\kol see Section~\ref{sec:last}}.
However, we cannot remove the assumption on the uniform bound on the order of {\kol cell stabilisers, as the
following example that we learned from Pierre-Emmanuel Caprace shows. Namely, the wreath product $G= A_5 \wr \mathbb{Z}$, where $A_5$ denotes the alternating group on $5$ elements, acts on a $\mathrm{CAT}(0)$ square complex with finite cell stabilisers \cite[Prop~9.33]{Gen}. However, $G$ neither contains a nonabelian free subgroup, nor is virtually solvable. 
}

Other classes of recurrent complexes arise from complexes with various combinatorial
nonpositive-curvature-like features. This includes Cayley complexes for the standard presentations
of Artin groups of extra-large type (see Subsection~\ref{s:Ar} for the definition).

\begin{main}
\label{main:Artin}
Let $X$ be the Cayley complex for the standard presentation of an Artin group $A_\Gamma$ of extra-large type. Suppose that $G$ acts {\kol almost freely} on~$X$ (e.g.\ $G=A_\Gamma$). Then any subgroup of $G$ is virtually cyclic, or virtually~$\mathbb{Z}^2$, or contains a nonabelian free group.
\end{main}

In Appendix A written jointly with Jon McCammond we extend Theorem~\ref{main:Artin} to a class of $2$-dimensional Artin groups
containing all large-type Artin groups. In the case where $G=A_\Gamma$ we extend Theorem~\ref{main:Artin} to all $2$-dimensional~$A_\Gamma$ with $W_\Gamma$ hyperbolic in \cite{MP}. We will give there an account on the current state of affairs concerning the Tits Alternative for other classes of Artin groups.

Another class of recurrent complexes arises from simply connected $B(6)$-small cancellation complexes (see Subsection~\ref{s:B6} for the definition and details).

\begin{main}
\label{main:B6}
Let $X$ be a simply connected $B(6)$-small cancellation complex. Suppose that $G$ acts {\kol almost freely} on $X$. Then any finitely generated subgroup of $G$ is virtually cyclic, or virtually $\mathbb{Z}^2$, or contains a nonabelian free group.
\end{main}

Let us note that Wise \cite{Wise2004} associated to each simply connected $B(6)$ complex~$X$ a $\mathrm{CAT}(0)$ cube complex $C$. Furthermore,
in \cite{SagWis2005} the Tits Alternative is shown for groups acting {\kol almost freely} on finite dimensional
$\mathrm{CAT}(0)$ cube complexes. However, the complex $C$ associated to
a simply connected $B(6)$ complex $X$ might not be finitely dimensional --- this happens e.g.\ when there is no bound on
the size of the $2$-cells in $X$. Therefore, the results from \cite{SagWis2005} do not imply Theorem~\ref{main:B6}.

The method of proving the Tits Alternative presented in this paper raises the following natural questions.
\begin{question}
Simplicial subdivisions of which $2$-dimensional combinatorial complexes can be metrised as recurrent complexes? Can it be done for $C(6)$-small cancellation complexes? What about Cayley complexes for standard presentations of $2$-dimensional Artin groups?
\end{question}

\medskip

\noindent \textbf{Idea of proof of the Main Theorem.}
For simplicity we assume that the action of $G$ is free and that $X$ is systolic. Supposing additionally that $X$ is countable, we exhaust the quotient $\overline X=X/G$ with compact subcomplexes $X_1\subset X_2\subset \cdots$. The fundamental groups $G_i$ of $X_i$ have direct limit $G$.

Collapsing we remove the free edges from $X_i$. We focus first on the case where some $X_i$ is \emph{thick} meaning that it has an edge $e$ of degree $\geq 3$. Consider the space of all local geodesics in $X_i$ that are concatenations of segments $\ldots, a_{-1}a_0, a_0a_1,a_1a_2,\ldots$ {\kol (where $a_i$ denote points),}
as indicated in Figure~\ref{f:0} on the right. We equip that space with a finite measure $\mu^*$ assigning to each `cylinder' of geodesics passing through prescribed consecutive segments  $a_0a_1,a_1a_2,\ldots,a_{n-1}a_n$ the value $\prod_{i=1}^{n-1} \frac{1}{\deg a_i-1}$, where $\deg a_i$ is the degree of the edge containing $a_i$. {\kol This is inspired by the work of Ballmann and Brin \cite{BallBrin1995}, who put a similar measure on a far larger space of geodesics.}
Using Poincar\'e recurrence \`a la \cite{BallBrin1995} we can find a local isometric embedding~$f$ of the dumbbell graph $\overline{\Gamma}$ (see Figure~\ref{f:0} on the left) into~$X_i$ with the following properties. Namely, $f$ sends the vertices of $\overline{\Gamma}$ into~$e$ and the edges of $\overline{\Gamma}$ into concatenations of the segments in Figure~\ref{f:0} on the right,
terminating perpendicularly to $e$. In particular, $f(\overline{\Gamma})$ avoids the vertices of~$X_i$ and hence the stabiliser in $G$ of the lift of $\overline{\Gamma}$ to $X$ contains a nonabelian free group.
\begin{figure}[h!]
	\centering
	\includegraphics[width=0.7368\textwidth]{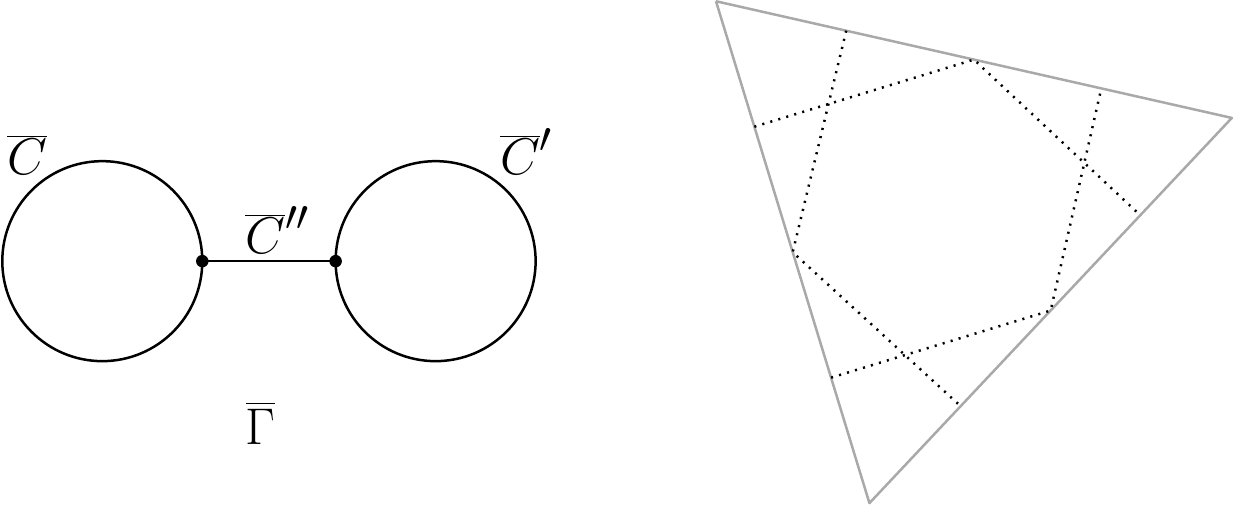}
	\caption{}
	\label{f:0}
\end{figure}

In the case where $X_i$ is not thick, it is a union of a graph and a dimension~$2$ pseudomanifold. The components of the pseudomanifold are $\pi_1$-injective in~$\overline X$, since otherwise attaching compressing discs puts us back in the thick case. If one such component is a hyperbolic surface $\Sigma$, we find a nonabelian free group in $\pi_1(\Sigma)< G$. Otherwise, each $G_i$ is a free product of some copies of $\Z,\Z^2$ and the Klein bottle group, which satisfies the Tits Alternative. One can arrange that the number of factors is bounded by a constant independent of $i$, and then use the Hopfian property to deduce that the sequence $G_1\to G_2\to\cdots$ stabilises. This shows that $G$ coincides with some $G_i$.

\smallskip

{\kol We believe that the overall method of our proof can be extended to treat all $2$-dimensional $\mathrm{CAT}(0)$ complexes. In particular, the `no thick subcomplexes' part (cf.\ Proposition~\ref{prop:OP}) is valid in such a general setting. To treat thick subcomplexes, one needs to find a method of `closing' geodesics without the use of the additional structure of the recurrent complex. It seems that finding free subgroups could work then also in higher dimensions, for analogues of thick subcomplexes. However, at the moment we do not know how to proceed in the `no thick subcomplexes' case in such higher dimensional setting, even for some restricted (combinatorial) classes of complexes, e.g.\ systolic complexes or Euclidean buildings.}

\smallskip

\noindent \textbf{Organisation.} In Section~\ref{sec:like} we define the main object of our interest, the recurrent complexes. We provide main examples and basic properties, and we show how to find nonabelian free subgroups given thick subcomplexes. In Section~\ref{s:inv} we treat the case where
there are no thick subcomplexes and we prove the Main Theorem.
In Section~\ref{s:B6Ar} we provide applications to $B(6)$-small cancellation complexes and Artin groups of extra-large type, proving Theorem~\ref{main:B6} and the finitely generated case of Theorem~\ref{main:Artin}. We discuss the case of infinitely generated subgroups and we complete the proofs of Theorems~\ref{main:cat0} and~\ref{main:Artin} in Section~\ref{sec:last}. {\kol In Appendix A written jointly with Jon McCammond we present the aforementioned extention of Theorem~\ref{main:Artin}.}
\medskip

\noindent \textbf{Acknowledgements.} We thank Pierre-Emmanuel Caprace, Sam Shepherd, and the anonymous referees for valuable comments. This paper was written while D.O.\ was visiting McGill University.
We would like to thank the Department of Mathematics and Statistics of McGill University
for its hospitality during that stay.

\section{Recurrent complexes}
\label{sec:like}
In this section we present a variant of the constructions introduced in \cite{BallBrin1995}.

\begin{definition}
\label{def:build-like}
Let $X$ be a triangle complex, i.e.\ a 2-dimensional simplicial complex built of geodesic Euclidean triangles, with
an action of a group~$G$. Let $x\in X^1-X^0$, let $e$ be the edge containing $x$ and let $T$ be a triangle containing $e$.
Then $\lk_x T$ denotes the open half-circle of directions at $x$ in $T$ that are transverse to $e$. By $\deg x$ we denote the degree of $e$, i.e.\ the number of triangles containing $x$.

For $v\in \lk_x T$, let $H(v)$ be the union of the directions $v'\in \lk_x T'$ with $T'\neq T$ such that there is a geodesic through $x$ in $T\cup T'$ with directions $v$ and $v'$. Note that for each triangle $T'$ containing $x$ with $T'\neq T$ there exists a unique such $v'$. Thus $|H(v)|=\deg x -1$. We have $v'\in H(v)$ if and only if $v\in H(v')$

Furthermore, for $v\in \lk_x T$ suppose that the geodesic in $T$ with the starting direction $v$ terminates at a point $x'\in X^1-X^0$. Then we denote its ending direction by $I(v)\in \lk_{x'}T$. Note that $I(I(v))=v$.

We say that $X$ is \emph{recurrent with respect to $G$} if there is a $G$-invariant subset $A$ of the union of all $\lk_x T$ such that all the following hold:
\begin{enumerate}[(i)]
\item
for each triangle $T$ the set of $a\in A$ with $a$ in some $\lk_x T$ is finite,
\item
for each $a\in A$ we have $H(a)\subset A$,
\item
for each $a\in A$ we have that $I(a)$ is defined and belongs to $A$,
\item
for each edge $e$ of degree $\geq 3$ there exists $x\in e$ such that for some (hence any by (ii)) triangle $T$ containing $e$ the direction in $\lk_x T$ perpendicular to $e$ belongs to $A$,
\item
there is no finite sequence $a_0,a_1,\ldots, a_n$, such that for all $0\leq i<n$ we have $a_{i+1}\in H(I(a_i))$, and $a_0=a_n$ or $a_0=I(a_n)$.
\end{enumerate}
\end{definition}

\begin{rem}
\label{rem:hereditary}
If a triangle complex $X$ is recurrent w.r.t.\ $G$ and $X'\subseteq X$ is a $G'$-invariant subcomplex, for some $G'< G$, then $X'$ is \emph{recurrent} w.r.t.\ $G'$.
\end{rem}

\begin{rem}
\label{rem:prop5}
If $X$ is $\mathrm{CAT}(0)$, then its local geodesics are global geodesics and hence embed, and consequently
Definition~\ref{def:build-like}(v) holds automatically for any $A$.
\end{rem}

\begin{exa}
\label{exa:1}
Suppose that $X$ admits a simplicial map $\rho$ to a simplicial complex consisting of only one triangle with angles $\frac{\pi}{2},\frac{\pi}{4},\frac{\pi}{4}$ such that $\rho$ restricted to each triangle of $X$ is an isometry. Then $X$ has $A$ satisfying Definition~\ref{def:build-like}(i)-(iv) w.r.t.\ any automorphism group of $X$. Indeed, it suffices to define $A\cap \lk_x T$ with $x$ in an edge $e$ to be
\begin{itemize}
\item the vector perpendicular to $e$, for $e$ the long edge and $x$ dividing $e$ in the ratio $1\colon 3$,
\item the vectors at angles $\frac{\pi}{4}$ to $e$, for $e$ the long edge and $x$ the midpoint of $e$,
\item the vectors at angles $\frac{\pi}{4}$ and $\frac{\pi}{2}$ to $e$, for $e$ the short edge and $x$ the midpoint of $e$,
\item empty otherwise.
\end{itemize}
In other words, $A$ is the union of the directions at the boundary of the two billiard trajectories in Figure~\ref{f:1a}.
\end{exa}

\begin{figure}[h!]
        \centering
        \includegraphics[width=0.90\textwidth]{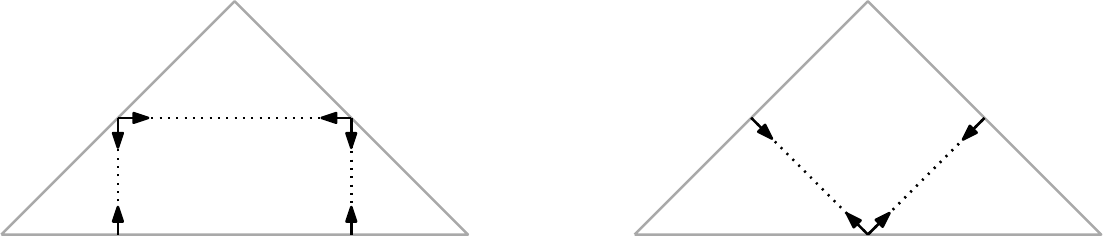}
        \caption{}
        \label{f:1a}
\end{figure}

\begin{exa}
\label{exa:2}
Suppose that $X$ admits a simplicial map $\rho$ to a simplicial complex consisting of only one triangle with angles $\frac{\pi}{2},\frac{\pi}{3},\frac{\pi}{6}$ such that $\rho$ restricted to each triangle of $X$ is an isometry. Then $X$ has $A$ satisfying Definition~\ref{def:build-like}(i)-(iv) w.r.t.\ any automorphism group of $X$. Indeed, it suffices to define $A$ in each triangle as the union of the directions at the boundary of the two billiard trajectories in Figure~\ref{f:2}, where $y$ is, say, the edge midpoint.
\end{exa}

\begin{figure}[h!]
        \centering
        \includegraphics[width=0.91\textwidth]{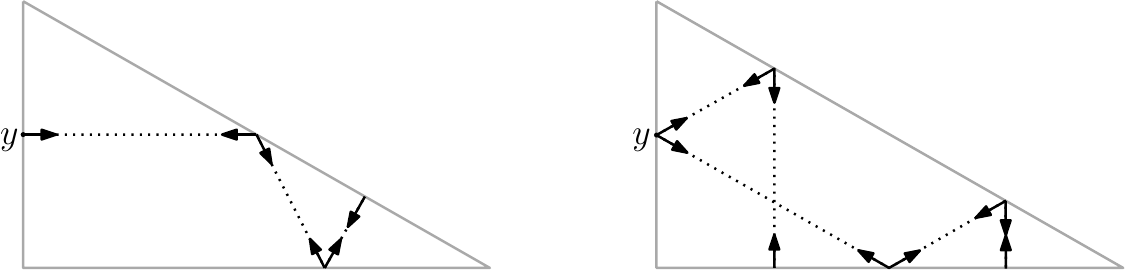}
        \caption{}
        \label{f:2}
\end{figure}

We have the following immediate consequence of Example~\ref{exa:2} and Remark~\ref{rem:prop5}.

\begin{cor}
\label{cor:sys}
Let $X$ be a $2$-dimensional systolic complex. Then the barycentric subdivision of $X$ is recurrent with respect to any automorphism group.
\end{cor}

\begin{cor}
\label{cor:build}
Let $X$ be a $2$-dimensional Euclidean building of type $W$ with its usual geometric realisation, where each chamber is realised as a Euclidean triangle of angles $\frac{\pi}{m_{st}}$ in the cases $W=\widetilde A_2, \widetilde C_2, \widetilde G_2$ or a square in the case $W=(\widetilde I_2)^2=D_\infty\times D_\infty$. Then $X$ has a subdivision $X^*$ that is recurrent with respect to any automorphism group $G$ of $X$, and such that $G$ acts on~$X^*$ without inversions.
\end{cor}
\begin{proof} Since $X$ is $\mathrm{CAT}(0)$, by Remark~\ref{rem:prop5} we have that Definition~\ref{def:build-like}(v) holds automatically. If $W=\widetilde G_2$, it suffices to take $X^*=X$ and use Example~\ref{exa:2}. If $W=\widetilde A_2$, we take $X^*$ to be the barycentric subdivision of $X$ and we use Example~\ref{exa:2} as well. If $W=\widetilde C_2$, let $X^*$ be obtained from $X$ by subdividing each triangle into two similar triangles along the altitude from the right angle. We then use Example~\ref{exa:1}. Finally, if $W=(\widetilde I_2)^2$, let $X^*$ be the barycentric subdivision of $X$ and use Example~\ref{exa:1}.
\end{proof}

\begin{rem}
\label{rem:subdivision}
In fact, if $X$ is a triangle complex with finitely many isometry types of `simplices with specified directions in $A$', recurrent w.r.t.\ $G$, then its barycentric subdivision $X'$ is also recurrent w.r.t.\ $G$, and consequently in the Main Theorem one can remove the assumption that $G$ acts without inversions.

Indeed, first note that for $\varepsilon$ sufficiently small, we can replace $A$ by $A'$ whose geodesic segments with starting direction $a'\in A'$ and ending direction $I(a')$ constitute the boundary of the $\varepsilon$-neighbourhood of analogous segments from $a\in A$ to $I(a)$. Secondly, except for finitely many of such $\varepsilon$, these segments from $a'\in A'$ to $I(a')$ do not pass through the vertices of $X'$, and hence they show the recurrence of $X'$ w.r.t.\ $G$.
\end{rem}

Recurrent complexes are designed to satisfy the following lemma.

\begin{definition}
\label{def:thick}
A $2$-dimensional simplicial complex is \emph{essential} if every edge has degree at least $2$, and none of connected components is a single vertex. An essential triangle complex is \emph{thick} if it has an edge of degree at least $3$.
\end{definition}

\begin{lemma}
\label{lem:thick}
Suppose that a triangle complex $X$ has all edges of finite degree, is thick, and is recurrent with $A/G$ finite. Then $G$ contains a nonabelian free group.
\end{lemma}

To prove Lemma~\ref{lem:thick} we will use the following method of \cite{BallBrin1995}.

\begin{definition} Suppose that an {\kol essential} triangle complex $X$ has all edges of finite degree, and is recurrent with $A/G$ finite.
Consider the Markov chain with states $A$ and the following transition function.
Let $b\in A\cap \lk_xT$. The transition probability $p(a,b)$ from $a\in A$ to $b$ equals $\frac{1}{\deg x-1}$ if $b\in H(I(a))$ and $0$ otherwise.
\end{definition}

\begin{rem}
\label{rm:measure}
A uniform measure $\mu$ on $A$ is stationary for that Markov chain.
Thus the space $A^\Z$ can be equipped with Markov measure $\mu^*$ invariant under the shift (see e.g.\ \cite[Ex~(8), page~21]{Wal}). Since $A/G$ is finite, the measure of the quotient $A^\Z/G$  by the diagonal action of $G$ is finite, w.l.o.g.\ a probability measure. Note that the shift map descends to $A^\Z/G$ and is still measure preserving.
\end{rem}

\begin{lemma}
\label{lem:Poincare} Suppose that an {\kol essential} triangle complex $X$ has all edges of finite degree, and is recurrent with $A/G$ finite.
Let $a\in A$ and $b\in H(a)$. Then there is a local geodesic $f\colon [0,l]\to X$ with the directions at $0,l$ mapping to $b,ga$ for some $g\in G$.
\end{lemma}
\begin{proof}
We have $p(I(a),b)\neq 0$. Thus by the Poincar\'e recurrence (see e.g. \cite[Thm~1.4]{Wal}) applied to $A^\Z/G$ we have a finite sequence $a_0=I(a),a_1=b,\ldots, a_n=ga_0$, for some $g\in G$, such that for all $0\leq i<n$ we have $a_{i+1}\in H(I(a_i))$. Define $f$ as the concatenation of the geodesics from $a_i$ to $I(a_i)$ for $1\leq i<n$ and from $a_n=ga_0$ to $gI(a_0)=ga$.
\end{proof}

\begin{rem}
\label{rm:Markov}
\kol{An alternative, more combinatorial way of proving Lemma~\ref{lem:Poincare} was suggested to us by Sam Shepherd. For example, if $G$ acts freely on~$X$, then the finite set $A/G$ is the set of states for appropriate Markov chain for which the uniform measure is stationary. Consequently, since we have a positive transition probability from the state $[I(a)]$ to the state $[b]$, we also have a positive probability of passing from $[b]$ to $[I(a)]$ after several steps.}
\end{rem}

\begin{proof}[Proof of Lemma~\ref{lem:thick}]
Let $\overline{\Gamma}$ (see Figure~\ref{f:0} left) be the graph obtained from joining the basepoints of two closed paths $\overline C,\overline C'$ by a path $\overline C''$ (their lengths will be determined later). Let $\Gamma$ be the universal cover of $\overline{\Gamma}$ with the action of the deck transformation group $F_2$. The main idea of the proof is to construct a homomorphism $\varphi\colon F_2\to G$ and $\varphi$-equivariant local isometry $\Gamma\to X-X^0$ that is injective on the set of directions at the vertices of $\Gamma$.

Let $C''$ be a lift to $\Gamma$ of $\overline C''$ with endpoints $c, c'$. Let $C,C'$ be some lifts of the paths $\overline C,\overline C'$ starting at $c, c'$. Let $h,h'\in F_2$ be the elements mapping $c$ to the other endpoint of $C$, and $c'$ to the other endpoint of $C'$, respectively. Observe that $C\cup C'\cup C''$ is a fundamental domain for the action of $F_2$ on~$\Gamma$. Thus to define an equivariant map $\Gamma \to X-X^0$ it suffices to define a homomorphism $\varphi \colon F_2 \to G$ and a map $f\colon C\cup C'\cup C''\to X-X^0$ with the property that $\varphi (h)$ maps $f(c)$ to the other endpoint of $f(C)$ and $\varphi(h')$ maps $f(c')$ to the other endpoint of $f(C')$.

Let $e$ be an edge of $X$ of degree $\geq 3$. Then for $i=1,2,3,$ there are distinct triangles $T_i$ containing $e$. Since $Y$ is recurrent, by Definition~\ref{def:build-like}(iv) we have $x\in e$ such that for any $T$ containing $e$ the direction in $\lk_x T$ perpendicular to $e$ belongs to $A$.
Let $v_i$ be that direction in $\lk_x T_i$.

Apply Lemma~\ref{lem:Poincare} to $a=v_1,b=v_2$, to obtain a local geodesic $f\colon [0,l]\to X-X^0$ with ending directions $v_2,gv_1$, for some $g\in G$. Identify $C$ with~$[0,l]$. Analogously, apply Lemma~\ref{lem:Poincare} to $a=v_1, b=v_3$ to obtain $f\colon C''\to X-X^0$ with ending directions $v_3,g''v_1,$ for some $g''\in G$. Finally, apply Lemma~\ref{lem:Poincare} to $a=v_2,b=g''v_3$, to obtain $f\colon C'\to X-X^0$ with ending directions $g''v_3,g'g''v_2,$ for some $g'\in G$. Define $\varphi \colon F_2\to G$ by $\varphi(h)=g$ and $\varphi(h')=g'$. By the observation above, we can extend $f\colon C\cup C'\cup C''\to X-X^0$ to a $\varphi$-equivariant map $\Gamma\to X-X^0$ for which we keep the same notation $f$.
Note that for each vertex $w$ of $\Gamma$ the three directions at $w$ are mapped under~$f$ to a $G$-translate of the triple $\{v_1,v_2,v_3\}$.

Let $E$ be the set of directed edges of $\Gamma$. Consider the map $f_*\colon E\to A$ that maps each directed edge $wu\in E$ to the direction of $f(wu)$ at $f(w)$. We will prove that $\varphi$ is injective by showing that $f_*$ is injective. Suppose that there are two edges $wu,w'u'\in E$ with $f_*(wu)=f_*(w'u')$. Then we also have $f_*(uw)=f_*(u'w')$, so without loss of generality we can assume that the embedded edge-path $\gamma$ in $\Gamma$ from $w$ to $w'$ passes through $u$. For $i=0,\ldots, n$, let $a_i$ be the images under $f_*$ of consecutive edges of $\gamma$. In particular, $a_0=f_*(wu)$ and in the case where $u'$ lies on $\gamma$ we have $a_n=f_*(u'w')$. If $\gamma$ does not contain $u'$, then we add $a_{n+1}=f_*(w'u')$. Note that for $i=0,1,\ldots$ we have $a_{i+1}\in H(I(a_i))$, and so
$f_*(wu)=f_*(w'u')$ contradicts Definition~\ref{def:build-like}(v).
\end{proof}

\section{Invariant cocompact subcomplexes}
\label{s:inv}

\begin{definition}
\label{d:subcplx}
Let $X$ be a simplicial complex with a simplicial action of a group~$G$. We say that a subcomplex $Z\subseteq X$ is
an \emph{\ics \ with respect to $G$} (shortly \emph{$G$-\sics}) if $Z$ is $G$-invariant, and the quotient $Z/G$
is compact. Note that a $G$-\sics\ is not required to be connected.
\end{definition}

A simplicial complex homeomorphic to the plane $\mathbb{R}^2$ (resp.\ to the $2$-sphere~$S^2$) is called
a \emph{simplicial plane} (resp.\ \emph{simplicial $2$-sphere}). A simplicial plane whose $1$-skeleton is Gromov-hyperbolic (w.r.t.\ to the metric where each edge has length $1$) is called \emph{hyperbolic}. By the classification of $2$-dimensional orbifolds, if $E$ is a non-hyperbolic simplicial plane with cocompact automorphism group $H$, then $H$ is virtually $\Z^2$. We call such $E$ a \emph{flat}.

\begin{lemma}
\label{l:noics}
Let $X$ be a simply connected $2$-dimensional simplicial complex with a finitely generated group $G$ acting {\kol almost freely} and without inversions. If each essential {$G$-\sics}\ in $X$ is a disjoint union of flats, then $G$ is virtually a free product of some number of $\Z$ and $\Z^2$.
\end{lemma}

\begin{proof}
Suppose first that $X$ is countable.
We may find an increasing sequence $X_1\subset X_2 \subset \cdots$ of connected {$G$-\sics}'s exhausting $X$. (Start with a $G$-orbit of a vertex, connect it equivariantly by edge-paths, then at each step add equivariantly remaining cells.) The action of $G$ on $X_i$ lifts to an action of a group $G_i$ on the universal cover $\widetilde X_i$ of $X_i$. The corresponding maps $\widetilde X_1\to \widetilde X_2 \to \cdots$ and $\widetilde X_i\to \widetilde X=X$ induce homomorphisms $G_1\to G_2\to\cdots$ and epimorphisms $G_i\twoheadrightarrow G$. Note that the vertex stabilisers of the action of $G_i$ on $\widetilde X_i$ coincide with the vertex stabilisers of the action of $G$ on $X_i$ and thus have uniformly bounded order.

Since $G$ is finitely generated, there is a finitely generated subgroup $H_1<G_1$ such that $H_1\to G$ is an epimorphism.
For each $i>1$, let $H_i$ be the image of $H_1$ under the the homomorphism $G_1 \to G_i$. We obtain an infinite sequence of epimorphisms
	\begin{align}\label{f:1}
	H_1 \twoheadrightarrow H_2 \twoheadrightarrow \cdots
	\end{align}
The epimorphism from the direct limit $\varinjlim H_i$ to $G$ is in fact an isomorphism. Indeed, let $h_i\in\ker (H_i\to G)$ and let $\alpha$ be a path
joining a basepoint $\widetilde x_i\in \widetilde X_i$ to $h_i\widetilde x_i$. The projection of $\alpha$ to $X_i$ is a closed path and it becomes contractible in some $X_j$, since $X_j$ exhaust $X$, and $X$ is simply connected. Consequently the image $h_j\in H_j$ of $h_i$ fixes the image of $\widetilde x_i$ in $\widetilde X_j$ and thus $h_j\in\ker (H_j\to G)$ implies $h_j=0$.

Since each essential {$G$-\sics}\ in $X$ is a disjoint union of flats and $G$ acts without inversions, every $X_i$ can be equivariantly collapsed (by removing triangles with free edges) to a space $Y_i$ that is a union of a graph and a disjoint union of flats. The preimage $\widetilde Y_i \subset \widetilde X_i$ of $Y_i$ under the covering map is thus a simply connected union of a graph and a disjoint union of flats, with a proper and cocompact action of~$G_i$. Let $\Gamma_i$ be the tree obtained from $\widetilde Y_i$ by quotienting each flat to a vertex. The quotient $\Gamma_i/G_i$ is a finite graph of groups $\mathcal G_i$ with $\pi_1\mathcal G_i=G_i$ and edge groups of uniformly bounded order. Its vertex groups $G_v$ are also finite of uniformly bounded order, or have the following description for a vertex $v$ obtained from quotienting a flat $Z$ to~$v$.
Namely, let $G'_v$ be the image of $G_v$ in the isometry group of~$Z$. We then have a short exact sequence $0\to K\to G_v\to G'_v\to 0$, with $K$ finite of uniformly bounded order. By the classification of $2$-dimensional Euclidean orbifolds, there are only finitely many possible isomorphism types for $G'_v$. Consequently, there are only finitely many possible isomorphism types for~$G_v$. Analogously, there are only finitely many possible isomorphism types for the subgroups of $G_v$.

If $H_1$ is generated by $d$ elements, then so is each $H_i$ for $i>1$. Since each $H_i$ is a subgroup of $G_i$, it is also the fundamental group of a finite graph of groups $\mathcal H_i$ with edge groups of uniformly bounded order. It follows, by a result of Linnell \cite[Thm 2]{Linnel1983}, that there is a uniform bound on the number of edges in a minimal such graph with fundamental group~$H_i$. (This is because the augmentation ideal in \cite[Thm 2]{Linnel1983} is generated by at most $d$ elements.) Edge and vertex groups of $\mathcal H_i$ have only finitely many isomorphism types. Furthermore, there are finitely many possible injections from edge groups to vertex groups, up to conjugations in vertex groups. However, such conjugations do not change $\pi_1 \mathcal H_i$. Hence there are only finitely many isomorphism types in $\{ H_i \}_{i=1}^{\infty}$.

Let $H$ be isomorphic to $H_i$ for infinitely many $i$. Note that $H$ is virtually a free product of some number of $\Z$ and $\Z^2$,
and thus it is residually finite. Moreover, $H$ is finitely generated, so it is Hopfian. It follows that if $H\cong H_i,H_{i+k}$, then $H_i\twoheadrightarrow H_{i+k}$ is an isomorphism, and hence, for every $j=i,i+1,\ldots, i+k-1$, the map $H_j\twoheadrightarrow H_{j+1}$ is an isomorphism. Therefore, the sequence~(\ref{f:1}) stabilises and $G=H$ is as required.

If $X$ is not countable, we consider the collection $X_\lambda$ of all connected {$G$-\sics} containing a given connected {$G$-\sics}\ $X_1$, which form a directed set under inclusion. We define appropriate $G_\lambda, H_\lambda$ as before, and we have again $G=\varinjlim H_\lambda$. There is still $H$ such that for every $X_\lambda$ there is $X_{\lambda'}\supset X_\lambda$ with $H_{\lambda'}\cong H$, and then replacing $X_1$ by such $X_{\lambda'}$ we obtain that all maps in our directed system are isomorphisms and consequently $G=H$.
\end{proof}

\begin{lemma}
\label{lem:injective}
Let $Z\subseteq X$ be a connected essential subcomplex in a connected simplicial complex $X$.
If $\pi_1Z\to \pi_1 X$ is not injective, then $Z$ is contained in a thick subcomplex $Z'\subseteq X$ with $Z'-Z$ finite.
\end{lemma}

Before we give the proof, we record the following consequences.

\begin{cor}
\label{cor:injective}
Let $Z\subseteq X$ be a connected essential {$G$-\sics}\ in a connected simplicial complex $X$. If $\pi_1Z\to \pi_1 X$ is not injective, then $Z$ is contained
in a thick {$G$-\sics}
\end{cor}

\begin{proof}
Apply Lemma~\ref{lem:injective} to $Z$. Then $GZ'$ is a thick {$G$-\sics}
\end{proof}

For a subcomplex $Z\subseteq X$ and a triangle $T$ of $Z$ let {\gal} denote the \emph{gallery connected component} of $Z$
containing $T$. That is, {\gal} is the subcomplex of $Z$ consisting of $T$ and
all the triangles in $Z$ that can be reached from $T$ by passing from a triangle to a triangle adjacent
along an edge.

\begin{cor}
\label{cor:essform}
Let $X$ be a simply connected simplicial complex that does not contain simplicial $2$-spheres.
Let $Z\subseteq X$ be an essential {$G$-\sics}\ that is not contained in a thick {$G$-\sics} Then for each triangle $T$ of $Z$, {\gal} is a simplicial plane.
\end{cor}

\begin{proof} Since $Z$ is essential and not thick, {\gal} is a 2-dimensional pseudomanifold. Thus {\gal} is homeomorphic with a connected surface with possibly some identifications on a discrete set of points. By Corollary~\ref{cor:injective}, {\gal} is simply connected, so it is homeomorphic with $S^2$ or $\R^2$.
\end{proof}

We now pass to the proof of Lemma~\ref{lem:injective}. A \emph{disc diagram} $D$ is a compact contractible simplicial complex with a fixed embedding in $\R^2$. Its \emph{boundary path} is the attaching map of
the cell at $\infty$. If $X$ is a simplicial complex, \emph{a disc diagram in $X$} is a nondegenerate simplicial map $\varphi \colon D\to X$, and its
\emph{boundary path} is the composition of the boundary path of $D$ and $\varphi$. We say that $\varphi$ is \emph{reduced} if it maps triangles sharing an edge to two distinct triangles. The \emph{area} of $\varphi$ is the number of triangles of $D$.

\begin{rem}
\label{rem:discs_exist}
Let $\gamma$ be a closed edge-path in a simplicial complex $X$. If $\gamma$ is contractible in $X$, then there is a reduced disc diagram in $X$ with boundary path $\gamma$. For $\gamma$ embedded, this is \cite[Lem 1.6]{JanSwi2006}. For $\gamma$ not embedded, attach a triangulated annulus $A$ to $X$ along $\gamma$ to obtain~$X'$. Applying the embedded case to the second boundary component $\gamma'$ of~$A$ we obtain a reduced disc diagram $\varphi' \colon D'\to X'$. Then $\varphi'^{-1}(X'-X)$ is the open combinatorial $1$-ball around $\partial D'$. Consequently, $\varphi'$ restricted to $D=D'-\varphi'^{-1}(X'-X)$ is a reduced disc diagram with boundary path $\gamma$.
\end{rem}

\begin{proof}[Proof of Lemma~\ref{lem:injective}]
Let $\varphi \colon D\to X$ be a minimal area reduced disc diagram with boundary path in $Z$ representing a nontrivial element of $\pi_1(Z)$. Note that $D$ cannot have area $0$. Let $Z'=Z\cup \varphi(D)$. Observe that $Z'$ is essential, since $Z$ is essential and $\varphi$ is reduced. Furthermore, let $T$ be a triangle of $D$ adjacent to a boundary edge $e$. By the minimality of area, $\varphi(e)$ has degree $\geq 3$ in~$Z'$.
\end{proof}

{\kol \begin{prop}
\label{prop:OP}
Let $X$ be a simply connected $2$-dimensional simplicial complex that contains no simplicial $2$-spheres with
a finitely generated group~$G$ acting almost freely and without inversions. If $X$ contains no thick {$G$-\sics}, then
$G$ is virtually cyclic, or virtually $\mathbb{Z}^2$, or contains a nonabelian free group.
\end{prop}}
\begin{proof}
Consider possible essential $G$-\sics\ $Z\subseteq X$. 
By Corollary~\ref{cor:essform}, for each triangle $T$ of $Z$ we have that $Y=\mathrm{Gal}_Z(T)$ is a simplicial plane. If any such~$Y$ is not a flat, then it is a hyperbolic simplicial plane. Consequently, the stabiliser $\mathrm{Stab}_G(Y)$, which acts properly and cocompactly on~$Y$, contains a nonabelian free group (by e.g.\
\cite[Thm~8.37]{GhydlH1990}). If $Y$ and $Y'$ are two such intersecting flats, then by Corollary~\ref{cor:injective}, the connected component $W$ of $Z$ containing $Y\cup Y'$ is an infinite valence tree of flats and thus $\mathrm{Stab}_G(W)$ contains a nonabelian free group. It remains to consider the case where each~$Z$ is a disjoint union of flats. Then by Lemma~\ref{l:noics}
we have that $G$ is virtually cyclic, or virtually $\Z^2$, or contains a nonabelian free group.
\end{proof}

\begin{proof}[Proof of the Main Theorem]
If $X$ contains a simplicial $2$-sphere $\Sigma$, then there is no triangle $T_0$ in $\Sigma$ with an $a_0\in A$ in some $\lk_{x_0} T_0$. Indeed, otherwise using Definition~\ref{def:build-like}(ii) and~(iii) we could construct in $A$ an infinite sequence $a_0,a_1, \ldots$ such that for all $i\geq 0$ we have $a_{i+1}\in H(I(a_i))$, with $a_i$ in some $\lk_{x_i}T_i$ and $T_i\subset \Sigma$. This would contradict Definition~\ref{def:build-like}(i) or~(v).
Consequently, by Definition~\ref{def:build-like}(iv) each edge of $\Sigma$ has degree 2 in $X$. We can thus remove all triangles and edges of each $\Sigma$ and replace them by a cone over the vertex set $\Sigma^0$. After this operation $X$ is still simply connected. We can thus assume that $X$ does not contain simplicial $2$-spheres.

If $X$ contains a thick $G$-\sics \ $Z$, then by Remark~\ref{rem:hereditary} the triangle complex~$Z$ is recurrent with respect to $G$. Since $G$ acts cocompactly on $Z$, we have that $A/G$ is finite. Moreover, since $G$ acts properly on $Z$, all edges of $Z$ have finite degree in $Z$.
Thus by Lemma~\ref{lem:thick} we have that $G$ contains $F_2$. If $X$ does not contain a thick {$G$-\sics}, then the Main Theorem follows from Proposition~\ref{prop:OP}.
\end{proof}

{\kol \begin{rem} The assumption in Proposition~\ref{prop:OP} that $X$ contains no simplicial $2$-spheres could be removed at the cost of allowing, instead of flats, trees of $2$-spheres and trees of a flat and $2$-spheres in the statement of Lemma~\ref{l:noics}. This would complicate slightly the proof of Lemma~\ref{l:noics}, so we decided to keep this assumption.
\end{rem}}

\section{More applications}
\label{s:B6Ar}

\begin{exa}
\label{exa:polygons}
Let $X$ be a $2$-dimensional combinatorial complex with all edges of length~$1$ and all $2$-cells regular Euclidean $2n$-gons, where $n$ might vary.
Suppose that all cells embed in $X$.
Then the barycentric subdivision~$X'$ of~$X$ is a triangle complex satisfying Definition~\ref{def:build-like}(i)-(iv) w.r.t.\ any automorphism group of $X$. Indeed, we define~$A$ in the triangles forming a given polygon as the directions coming from the union of line segments perpendicular to opposite edge pairs and dividing them in the ratio $1\colon 3$, see Figure~\ref{f:3}.
\end{exa}

\begin{figure}[h!]
	\centering
	\includegraphics[width=0.65\textwidth]{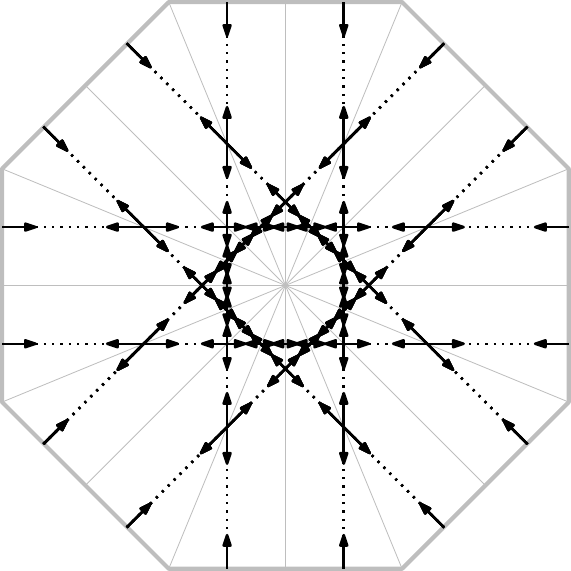}
	\caption{}
	\label{f:3}
\end{figure}

Here we study two classes of complexes, where we have also Definition~\ref{def:build-like}(v), and consequently $X'$ is recurrent w.r.t.\ any automorphism group of $X$.

\subsection{$B(6)$ complexes}
\label{s:B6}

The following notion was introduced by Wise \cite{Wise2004}. Let $X$ be a $2$-dimensional combinatorial complex. $X$ satisfies the \emph{$B(6)$ ({small cancellation}) condition} if for each $2$-cell $R$, and for each path $S\to \partial R$ which is the concatenation of at most $3$ pieces, we have $|S|\leq |\partial R|/2$, where $|\cdot |$ denotes the number of edges in a path. (See \cite{Wise2004} for definitions of paths, pieces, and further details.) In particular, $B(6)$ complexes satisfy the combinatorial $C(6)$ small cancellation condition \cite[Prop~2.11]{Wise2004}.

\begin{exa}
\label{exa:B6}
Suppose that $X$ is simply connected and satisfies $B(6)$. By subdividing each edge into two, we can assume that for every $2$-cell $R$ of~$X$ the length $|\partial R|$ is even.
A \emph{hypergraph} in $X$ is then a connected component~$\Lambda$ of a graph whose vertices correspond to edges of $X$ and whose edges correspond to pairs of antipodal edges in $R$, with the obvious map $\Lambda \to X$ \cite[Def 3.2 and Rem 3.4]{Wise2004}. Equip the barycentric subdivision $X'$ of $X$ with the metric and $A$ of Example~\ref{exa:polygons}. Note that all the $2$-cells of $X$ embed by \cite[Cor~2.9]{Wise2004}.

Observe that any sequence $a_0,a_1,\ldots, a_n$ of elements of $A$, with $a_{i+1}\in H(I(a_i))$ for $0\leq i<n$ extends to such a sequence with $a_0\in \lk_xT, a_n\in \lk_{x'}T'$ with $x,x'$ in the edges of $X$ (and dividing them in the ratio $1\colon 3$). Joining consecutive $a_i$ by geodesics we obtain a local geodesic segment $\gamma\to X$, which factors (up to a distance $\frac{1}{4}$ translation) through a hypergraph $\Lambda\to X$.
Since $\Lambda\to X$ is an embedding \cite[Cor 3.12]{Wise2004}
we have Definition~\ref{def:build-like}(v), and thus $X'$ is recurrent with respect to any automorphism group of $X$. Consequently
the Main Theorem applies to $X'$, implying Theorem~\ref{main:B6}.
\end{exa}

Our arguments do not extend directly to $C(6)$-small cancellation complexes because, as shown in \cite[\S3.5]{Wise2004}, in that case a hypergraph might not embed. It is an open question (see e.g.\ \cite[Prob 1.4]{Wise2004}) whether one can define any reasonable `walls' in that case.

\subsection{Artin groups of extra-large type}
\label{s:Ar}

Let $\Gamma$ be a finite simple graph with each of its edges labeled by an integer $\ge 2$. Let $V\Gamma$ be the vertex set of $\Gamma$. The \emph{Artin group} $A_{\Gamma}$ is given by the following presentation, where $p_m(a,b)$ denotes the word $\underbrace{aba\cdots}_{m}$:
\begin{center}
	$\langle V\Gamma\ |\ p_m(a,b)=p_m(b,a)$ for each edge $ab$ labelled by $m \rangle$.
\end{center}

We call the presentation above the \emph{standard presentation} for $A_{\Gamma}$. The Artin group $A_{\Gamma}$
is \emph{of extra-large type}
if all the labels satisfy $m\ge 4$. The \emph{Coxeter group $W_\Gamma$} is obtained from $A_\Gamma$ by adding the relations $a^2=1$ for all $a\in V\Gamma$.

Let $X$ be the Cayley complex of $A_{\Gamma}$ for the standard presentation.
It consists of cells that are $2m$-gons with $m$ the labels of $\Gamma$.
A \emph{hypergraph} $\Lambda\to X$ is defined as in Example~\ref{exa:B6}. Note that $\Lambda\to X$ is an embedding, since it projects to a hypergraph in the Cayley complex of $W_\Gamma$, which is embedded. Similarly, $X-\Lambda$ has two connected components.

\begin{prop}
	\label{l:hypgr}
Suppose that $A_\Gamma$ is of extra-large type. Then each hypergraph $\Lambda$ in $X$ is a tree.
\end{prop}

By Example~\ref{exa:polygons} and Proposition~\ref{l:hypgr}, the Main Theorem applies to $X'$ implying Theorem~\ref{main:Artin} for finitely generated subgroups of $G$.

\smallskip

The remaining part of the section is devoted to the proof of Proposition~\ref{l:hypgr}.

\begin{lemma}
	\label{l:block}
Let $\Gamma$ be a single edge $ab$ with label $m\geq 3$, and let $\Lambda\subset X$ be a hypergraph.
Suppose that we have an edge-path $\gamma$ in $X$ with only the first and the last edge corresponding to vertices of $\Lambda$. Then $\gamma$ is not labelled by a word of the form $a^kba^l, b^kab^l$ or $a^kb^l$ with $k,l\in \Z-\{0\}$.
\end{lemma}	
\begin{proof} Let $w\in W_\Gamma$ denote the element represented by the longest word $p_m(a,b)$. Consider an edge of~$\Lambda$. Its vertices correspond to opposite edges $e,e'$ in a $2$-cell $D$ of $X$. Thus $e$ and~$e'$ are labelled by the same (resp.\ distinct) letters for $m$ even (resp.\ odd) and oriented towards the same side of $\Lambda$. A path in $\partial D$ starting with $e$ and ending with $(e')^{-1}$ is labelled by a word projecting to $wa$ or $wb$ in~$W_\Gamma$. If we consider another edge of~$\Lambda$ with vertices corresponding to $e',e''$, a path starting with $e$ and ending with $(e'')^{-1}$ is labelled by a word projecting to $1\in W_\Gamma$. Consequently, a concatenation of such paths that is starting with (an edge labelled by) $a$ and ending with~$a^{-1}$ is labelled by a word projecting to $1$ or possibly to $wa$ for $m$ even, which is distinct from the possible projections $aba,ab,ba,b$ of $a^kba^l$. Similarly, if such a concatenation of paths is starting with $a$ and ending with~$b^{-1}$, then $m$ is odd, and the path is labelled by a word projecting to $wb$, which is distinct from the possible projections $ab,a,b,1$ of $a^kb^l$.
\end{proof}

\begin{lemma}
\label{lem:corner}
Let $\Gamma$ be a single edge $ab$ with label $m\geq 3$. Let $u$ be a cyclically reduced word representing $1\in A_\Gamma$. Then, possibly after a cyclic permutation of $u$, there are two subwords $w_1,w_2$ of $u$ of forms $p_{m}(a,b),p_{m}(a,b)^{-1},p_{m}(b,a),$ or $p_{m}(b,a)^{-1}$, none of whose letters lie in the same syllable of $u$, nor any of the cyclic permutations of $u$.
\end{lemma}

In the proof we will use the small cancellation techniques of \cite{AppSch1983}. Here we recall only the notions that are less standard and we refer the reader to \cite{AppSch1983} for more details. An \emph{$\mathcal R$-diagram} $M$ is a van Kampen diagram, with the boundary path of each $2$-cell (called a \emph{region}) labelled by a word in a set of relators~$\mathcal R$. If we ignore the labelling, $M$ is just called a \emph{diagram}. A \emph{spike} of~$M$ is a boundary vertex of valence $1$.
The \emph{interior degree} $i(D)$ of a region $D$ is the number of interior edges of $\partial D$ (after forgetting vertices of valence $2$). $D$ is a \emph{simple boundary region} if its \emph{outside boundary} $\partial D\cap \partial M$ is nonempty, and $M-\overline D$ is connected. A \emph{singleton strip} is a simple boundary region with $i(D)\leq 1$. A \emph{compound strip} is a subdiagram $R$ of $M$ consisting of regions $D_1,\ldots, D_n,$ with $n\geq 2$ with $D_{k-1}\cap D_k$ a single interior edge of $R$ (after forgetting vertices of valence $2$), satisfying $i(D_1)=i(D_n)=2,i(D_k)=3$ for $1<k<n$ and $M-R$ connected.

We will use the following \cite[Lem 2]{AppSch1983}, where case~(ii) needs to be added to account for a minor error in the third paragraph of their proof where the singleton strip $D$ might be glued to a region contained in two distinct strips of $M'$.
\begin{lemma}
\label{lem:strips}
Let $M$ be a simply connected diagram with no spikes and more than one region. If $M$ satisfies $C(4)$ and $T(4)$, then
\begin{enumerate}[(i)]
\item
$M$ contains two singleton strips, or
\item
$M$ contains one singleton strip and two compound strips, or
\item
$M$ contains four compound strips.
\end{enumerate}
\end{lemma}

\begin{proof}[Proof of Lemma~\ref{lem:corner}]
Let $M$ be a minimal $\hat {\mathcal R}_{ab}$-diagram for $u$, where $\hat {\mathcal R}_{ab}$ is the symmetrised set obtained from the relator $p_m(a,b)(p_m(b,a))^{-1}$. Each region of $M$ has two \emph{separating vertices} separating the paths in its boundary labelled $p_m(a,b)$ and $(p_m(b,a))^{-1}$. We prove inductively on the number of regions of $M$ a slight strengthening of Lemma~\ref{lem:corner} where we add the requirement that each $w_i$ labels a path in the outside boundary of a region of~$M$ (which joins its separating vertices).

If $M$ has a single region, then (after possibly a cyclic permutation and/or an inversion) we have $u=p_m(a,b)(p_m(b,a))^{-1}$ and we are done. Otherwise we can apply Lemma~\ref{lem:strips}, since by \cite[Lem~3]{AppSch1983}, $M$ satisfies $C(4)$ and $T(4)$.

In cases (ii) and (iii) we claim that there are regions $D_1,D_2$ in the strips of~$M$ with both of their separating vertices on their outside boundaries (we will call such regions \emph{exposed}), and such that the outside boundaries of $D_1,D_2$ are separated in $\partial M$ by outside boundaries of other simple boundary regions of interior degree at most~$2$.

To justify the claim, by \cite[Lem 5]{AppSch1983} each strip $S$ has an exposed region~$D$. Consequently, in case (ii), if we denote by $D_s$ the singleton strip and by $S,S'$ its consecutive (clockwise and counterclockwise) compound strips, there are exposed regions $D$ in $S$ and $D'$ in $S'$. We can take $D_1=D_s$ and $D_2$ to be $D$ or $D'$ unless $D$ and $D'$ are the regions consecutive to $D_s$ among the interior degree $2$ regions of $M$ lying in $S$ and $S'$. But then we can take $D_1=D, D_2=D'$ instead. Finally, in case (iii), if we call any exposed region $D_s$, we have two strips $S,S'$ disjoint from $D_s$ and the same procedure as in case (ii) yields the required $D_1,D_2$. This justifies the claim.

Choose $w_1,w_2$ labelling the paths in the outside boundaries of $D_1,D_2$ between their separating vertices. Note that each simple boundary region of interior degree at most~$2$ has outside boundary of length at least $2$, and hence witnesses a syllable change in $u$. Thus by the claim the letters of $w_1,w_2$ do not lie in the same syllable.

It remains to consider case (i). Let $D$ be a singleton strip and let $M'$ be the diagram obtained from $M$ by removing $D$ and possibly a spike or a sequence of spikes, so that $M'$ has cyclically reduced boundary word $u'$. By induction hypothesis, $u'$ has appropriate subwords $w_1',w_2'$ in the outside boundaries of single regions. We can choose $w_1=w_1', w_2=w_2'$ for $u$, unless the intersection $\beta$ of $D$ or the final spike with $M'$ contains an interior vertex of the path $\alpha$ labelled by one of the $w'_i$, say $w'_1$. There is an endpoint $x$ of~$\beta$ that is not an endpoint of $\alpha$. Let $w_2=w_2'$. For the choice of $w_1$, suppose first that $\beta$ is not a single vertex. Then $x$ is a separating vertex of~$D$. Thus we can take $w_1$ to be the word labelling the length $m$ subpath of $\partial D -\beta $ starting at $x$. It remains to consider $\beta=x$. In that case we can take $w_1$ to be the word labelling any path (there might be two) in the outside boundary of $D$ joining its separating vertices.
\end{proof}

We are finally ready for the following.

\begin{proof} [Proof of Proposition~\ref{l:hypgr}] We first focus on the case where $\Gamma$ is a single edge. {\kol (For future reference in the proof of Proposition~\ref{l:hypgr2}, note that the following argument works also for $m=2,3$.)}
Suppose that there is a cycle $\gamma$ in $\Lambda$ with edges corresponding to $2$-cells $D_0,D_1,\ldots,D_k=D_0$ and vertices corresponding to edges $e_0,e_1,\ldots, e_k=e_0$ of $X$ with $e_{i-1},e_{i}$ opposite in $D_{i}$ for $i=1,\ldots, k$. Let $u_i$ be the labels on the length $m-1$ paths joining in $\partial D_i$ either the initial vertices or the terminal vertices of the directed edges $e_{i-1},e_{i}$. Then $u=u_1u_2\cdots u_k$ represents the trivial element in $A_\Gamma$ and hence there is an $\hat {\mathcal R}_{ab}$-reduced diagram $M$ for $u$. Suppose that $M$ has minimal possible number $n$ of regions among all cycles $\gamma$ of $\Lambda$.

Attach to $M$ along its boundary all the $D_i$, and glue $D_i$ to $D_{i+1}$ along $e_i$, to form a diagram $M'$. Note that $M'$ is still reduced, since if $D_i$ would cancel with a region $D$ of $M$, we would have (using the observation that $\Lambda$ does not self-intersect) a cycle $\gamma'$ of $\Lambda$ inside $M$, contradicting the minimality of~$n$. Furthermore, $D_i$ cannot cancel with $D_j$, since this would also contradict the minimality of~ $n$. By Lemma~\ref{lem:strips} and \cite[Lem 5]{AppSch1983}, there is $D_i$ with its outside boundary of length at least $m$, which contradicts the definition of a hypergraph. This finishes the case where $\Gamma$ is a single edge.

Now we consider arbitrary $\Gamma$. By \cite{Lek}
for each edge $ab$ of $\Gamma$ the group $A_{ab}$ embeds in $A_\Gamma$. Thus we have in $X$ copies of the Cayley complex $X_{ab}$ of $A_{ab}$, which we call \emph{blocks}, corresponding to the cosets of $A_{ab}$ in $A_\Gamma$.
Given a cycle $\gamma$ in $\Lambda$, define $e_i,D_i$ as before. Let $B_0,B_1,B_2,\ldots, B_m=B_0$ be the consecutive blocks visited by $(D_i)$. Note that $m>0$ by the case of $\Gamma$ a single edge. For $j=1, \ldots, m,$ let $D_{i_j}\subset B_{j-1}, D_{i_j+1}\subset B_j$ be the cells where we transition from $B_{j-1}$ to $B_j$. Let $L_{j}$ be the line that is the connected component of $B_{j-1}\cap B_{j}$ containing $e_{i_j}$ (in fact $L_{j}=B_{j-1}\cap B_{j}$ but we do not need that). Consider closed immersed edge-paths $\delta=\delta_1\delta_2\cdots \delta_m$ in $X$ such that each $\delta_j$ is a path in $B_j$ from $L_{j}$ to $L_{j+1}$ (where $L_{m+1}=L_1$). Note that each $\delta_j$ is nontrivial since otherwise a path in $L_{j}\cup L_{j+1}$ labelled by $a^kb^l$ would contradict Lemma~\ref{l:block}. The word $v=v_1v_2\cdots v_m$, with $v_j$ the label of $\delta_j$, represents the trivial element in $A_\Gamma$ and hence there is an $\mathcal{R}$-reduced diagram $M$ for $v$, where $\mathcal{R}$ is the symmetrised set obtained from the standard presentation of $A_\Gamma$. Choose $\delta$ so that $M$ has the minimal possible number of regions.

By \cite[Lem 8]{AppSch1983} $M$ satisfies $C(6)$, and so if $M$ has more than one region, it has a simple boundary region $D$ with interior degree at most $3$. Suppose w.l.o.g.\ that $\partial D$ is labelled by a word in $a$ and $b$. Since the words labelling the intersections of $\partial D$ with its adjacent regions cannot exceed one syllable, by Lemma~\ref{lem:corner} there is an occurrence of $p_{m-1}(a,b)$ or $p_{m-1}(b,a)$ in the outside boundary of $D$. Since $m\geq 4$, this shows that $v$ has a syllable consisting of a single letter $a$ or $b$, say $b$. Suppose that this $b$ occurs in $v_j$. By the minimality of $M$, we have that $B_j$ is not a copy of $X_{ab}$ and so $\delta_j$ consists of a single edge. This contradicts Lemma~\ref{l:block} applied to $B_j$. If $M$ is a single region, the proof is analogous.
\end{proof}

\section{Infinitely generated groups}
\label{sec:last}

In this section we explain when we can extend the Main Theorem to infinitely generated $G$.

{\kol Let $G$ be a group acting on a $\mathrm{CAT(0)}$ triangle complex $X$ with finitely many isometry types of simplices.
By \cite[{I.7.19}]{BriHaf1999}, we have that $X$ is complete, and so by \cite[{II.2.8}]{BriHaf1999} every finite subgroup of $G$ fixes a cell of $X$. Consequently, if $G$ acts almost freely, then there is a bound on the order of finite subgroups of $G$. Since $G$ acts properly in the sense of \cite[I.8.2]{BriHaf1999},} the following lemma is an immediate consequence of \cite[{II.7.5 and II.7.7(2)}]{BriHaf1999}.

\begin{lemma}
	\label{l:flattorus}
	Let $G$ be a group acting {\kol almost freely} on a $\mathrm{CAT(0)}$ triangle complex with finitely many isometry types of simplices.
	Then every sequence $G_1<G_2<\cdots$ of virtually abelian subgroups of $G$
	stabilises, that is, there is $n$ such that for all $i\geq n$ the inclusion $G_i<G_{i+1}$ is an isomorphism.
\end{lemma}

In view of Corollaries~\ref{cor:sys} and~\ref{cor:build}, the following completes the proof of Theorem~\ref{main:cat0}.

\begin{cor} If $X$ is a $\mathrm{CAT(0)}$ triangle complex with finitely many isometry types of simplices, then the Main Theorem holds also for infinitely generated~$G$.
\end{cor}

\begin{proof} Consider the family of finitely generated
subgroups $G_\lambda$ of $G$. If any $G_\lambda$ contains $F_2$, then we are done. Otherwise, by the Main Theorem every $G_\lambda$ is virtually $\mathbb{Z}^2$ or virtually cyclic. It remains to observe that for some~$\lambda$ we have $G_\lambda=G$. Indeed, otherwise we can inductively define a sequence $G_{\lambda_1}\lneq G_{\lambda_2}\lneq\cdots$ contradicting Lemma~\ref{l:flattorus}.
\end{proof}

In \cite{Wise2003} Wise presents a procedure of constructing a systolic complex associated to
every simply connected $B(6)$ complex (in fact, to every simply connected $C(6)$ complex). This construction is described in details in \cite{OsaPry2018} and we follow the notation from there.
Without loss of generality we may assume that $X$ is the union of its $2$-cells: otherwise we attach equivariantly a $2$-cell to every edge of degree~$0$.
Then, the \emph{Wise complex} $W(X)$ of $X$ is defined as the nerve of the covering of $X$ by closed $2$-cells.
The Wise complex of a simply connected $C(6)$ complex is systolic, see \cite[Thm~6.7]{Wise2003} and \cite[Thm~7.10]{OsaPry2018}.
\begin{lemma}[{\cite[Lem 2.2]{Prytula2018}}]
	\label{l:flattorus_B6}
	Let $G$ be a group acting properly on a uniformly locally finite systolic complex.
	Suppose that there is a bound on the order of finite subgroups of $G$.
	Then every sequence $G_1<G_2<\cdots$ of virtually abelian subgroups of $G$
	stabilises.
\end{lemma}	

{\kol Similarly as in the $\mathrm{CAT(0)}$ setting, finite subgroups of isometries of systolic complexes fix points \cite{ChOs}. Therefore, a group acting almost freely on a systolic complex has a bound on the order of finite subgroups.  Consequently, Lemma~\ref{l:flattorus_B6} implies the following:
}

\begin{cor} If $X$ is a simply connected $B(6)$ complex with uniformly locally finite $W(X)$, then Theorem~\ref{main:B6} holds also for infinitely generated~subgroups of $G$.
\end{cor}

We believe that Theorem~\ref{main:B6} holds for infinitely generated subgroups of $G$ without the assumption of the uniform local finiteness of $W(X)$.
The reason for that assumption in Lemma~\ref{l:flattorus_B6} is that it is deduced from the systolic Flat Torus Theorem proved at the moment only for uniformly locally finite systolic complexes \cite[Thm 6.1]{Elsner2009}.

Finally, we complete the proof of Theorem~\ref{main:Artin}. By \cite[Thm 5.6]{HuaOsa2017}, a group acting properly on the Cayley complex for the standard presentation of an Artin group of extra-large type acts properly on a uniformly locally finite systolic complex. Therefore,
we can also apply Lemma~\ref{l:flattorus_B6} to 
{\kol deduce the case of infinitely generated subgroups in Theorem~\ref{main:Artin} from the finitely generated case covered in Section~\ref{s:Ar}.}

\appendix
\label{s:app}
\section{When hypergraphs are trees for $2$-dimensional Artin groups\\ by Jon McCammond, Damian Osajda, and Piotr Przytycki}

In this appendix, we generalise Proposition~\ref{l:hypgr} to the following.
Here $X$ is the Cayley complex for the standard presentation of an Artin group $A_\Gamma$.

\begin{prop}
	\label{l:hypgr2}
Suppose that $A_\Gamma$ is a $2$-dimensional Artin group. Then each hypergraph $\Lambda$ in $X$ is a tree if and only if $\Gamma$ has no triangle with an edge labelled by $2$.
\end{prop}

By Example~\ref{exa:polygons} and Proposition~\ref{l:hypgr2}, the Main Theorem applies to $X'$ implying the following.

\begin{theorem}
\label{main:Artin2}
Let $A_\Gamma$ be a $2$-dimensional Artin group such that $\Gamma$ has no triangle with an edge labelled by $2$.
Suppose that $G$ acts {\kol almost freely} on $X$. Then any finitely generated subgroup of $G$ is virtually cyclic, or virtually $\mathbb{Z}^2$, or contains a nonabelian free group.
\end{theorem}

Note that if additionally $\Gamma$ has no square with at least three edges labelled by $2$, then by \cite[Thm~5.6]{HuaOsa2017} $G$ acts properly on a uniformly locally finite systolic complex and therefore we can apply  Lemma~\ref{l:flattorus_B6} to show that Theorem~\ref{main:Artin2} holds also for $G$ infinitely generated.

We first justify the `only if' part of Proposition~\ref{l:hypgr2}.

\begin{exa}
\label{exa:app}
Suppose that $\Gamma$ has a triangle $abc$ with the edge $ab$ labelled by~$2$. Consider the following $\mathcal{R}$-reduced diagram $M$ consisting of $12$ regions. The $4$ central regions have boundaries labelled by $aba^{-1}b^{-1}$. The $4$ top and bottom regions have boundaries labelled by $a$ and $c$, and the $4$ left and right regions have boundaries labelled by $b$ and $c$ (see Figure~\ref{f:exA3}). $M$ contains a cycle of $\Lambda$, which is thus not a tree.
\end{exa}

\begin{figure}[h!]
	\centering
	\includegraphics[width=0.75\textwidth]{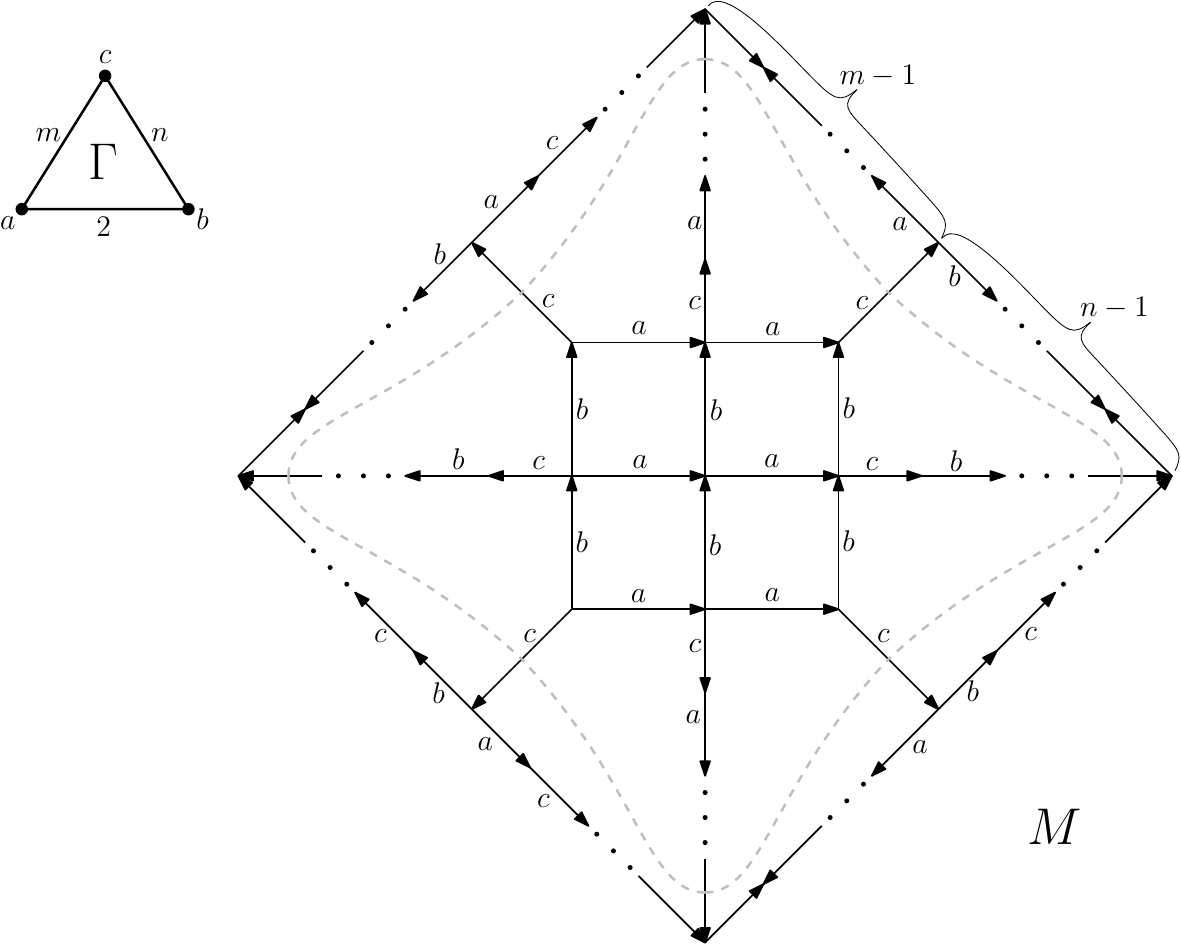}
	\caption{Example~\ref{exa:app}. A cycle in $\Lambda$ is marked by the dashed line.}
	\label{f:exA3}
\end{figure}

For the `if' part of Proposition~\ref{l:hypgr2}, we need the following.

\begin{lemma}
\label{lem:app}
Let $\Sigma$ be the Cayley complex of the Coxeter group $W_\Gamma$ such that $\Gamma$ has no triangle with an edge labelled by $2$. Then for any $2$-cells $\sigma, \tau $ of~$\Sigma$ sharing an edge, and a hypergraph $\Lambda_\sigma$ intersecting $\sigma$, there is a hypergraph $\Lambda_\tau$ intersecting $\tau$ that is disjoint from or equal to $\Lambda_\sigma$.
\end{lemma}

\begin{proof} Denote $e=\sigma\cap \tau$. Note that we can assume that $\Lambda_\sigma$ is disjoint from~$e$, since otherwise we can take $\Lambda_\tau=\Lambda_\sigma$.

Suppose first that $\tau$ is a square. Then take $\Lambda_\tau$ to be the hypergraph intersecting $\tau$ but not $e$. Let $f$ be an edge of $\sigma$ sharing a vertex $v$ with $e$. If $\Lambda_\sigma$ intersects $f$, then let $\alpha$ be the path obtained by concatenating at $v$ a half-edge of $f$  and a half-edge in $\tau$ ending at $\Lambda_\tau$. Note that $\alpha$ is a geodesic in the Moussong metric on $\Sigma$, by the assumption that $\Gamma$ has no triangle with an edge labelled by $2$. Moreover, the Alexandrov angle between the endpoints of $\alpha$ and $\Lambda_\sigma$ and $\Lambda_\tau$ are $\frac{\pi}{2}$. Consequently $\Lambda_\sigma$ and $\Lambda_\tau$ are disjoint.
Analogously, for the hyperplane $\Lambda'_\sigma$ intersecting the other edge of $\sigma$ sharing a vertex with $e$, we have that $\Lambda'_\sigma$ and $\Lambda_\tau$ are disjoint. Consequently, $\Lambda_\tau$ is contained in the same component of $\Sigma-\Lambda_\sigma\cup \Lambda'_\sigma$ as $e$. Then $\Lambda_\tau$ is also disjoint from all other hyperplanes intersecting $\sigma$ but not $e$.

If $\sigma$ is a square, then the same argument shows that we can take $\Lambda_\tau$ to be any hyperplane intersecting $\tau$ but not $e$.

It remains to consider the case where neither of $\sigma, \tau$ are squares.
We modify the Moussong metric on $\Sigma$ in the following way.
Every square remains a Euclidean square of side length $1$. Every $2$-cell that is not a square is subdivided
into triangles by (Moussong) geodesic segments joining vertices with the centre and we turn each triangle into
the Euclidean equilateral triangle of side length $1$.
Note that because $\Gamma$ has no triangle with an edge labelled by~$2$, this metric is still CAT(0).
As before, let $f$ be an edge of $\sigma$ sharing a vertex $v$ with $e$.
Let $g$ be the edge of $\tau$ sharing a vertex with $e$ distinct from $v$, and let $\alpha$ be the path obtained by concatenating at $v$ a half-edge of $f$ and the geodesic segment from $v$ to the centre of $\tau$.
If $\Lambda_\sigma$ intersects~$f$, then take $\Lambda_\tau$ to be the hypergraph intersecting $g$.
Again $\alpha$ is a geodesic meeting $\Lambda_\sigma$ and (since we modified the metric) $\Lambda_\tau$ at Alexandrov angle $\frac{\pi}{2}$. Consequently $\Lambda_\sigma$ and $\Lambda_\tau$ are disjoint.

Finally, consider a hyperplane $\Lambda'_\sigma$ intersecting $\sigma$ but none of its edges sharing vertices with $e$.  Let $w$ be the centre of $\sigma$, and let $\lambda_\sigma$ be the component of $\Lambda_\sigma-w$ intersecting $f$. Let $\lambda'_\sigma$ be the component of $\Lambda'_\sigma-w$
intersecting $\partial \sigma$ earlier if we traverse it starting from $e$ and ending with $f$. Since we modified the metric, $\Lambda''=\lambda_\sigma\cup w \cup \lambda'_\sigma$ is convex. Using the  path $\alpha$  we obtain that $\Lambda''$ is disjoint from $\Lambda_\tau$. Since $\Lambda'_\sigma$ is contained in the closure of the component of $\Sigma-\Lambda''$ that does not contain $e$, we have that
$\Lambda'_\sigma$ and $\Lambda_\tau$ are disjoint.
\end{proof}

\begin{proof}[Proof of Proposition~\ref{l:hypgr2}]
The `only if' part follows from Example~\ref{exa:app}. For the `if' part, given a cycle $\gamma$ in $\Lambda$, define $D_i, B_j, \delta_j$ and $M$ as in the proof of Proposition~\ref{l:hypgr}. Note that $\delta_j$ are still nontrivial, since Lemma~\ref{l:block} obviously holds for $m=2$ and the word $a^kb^l$.

We claim that $M$ has no $2$-cells. Indeed, otherwise let $\tau_X$ be a $2$-cell of~$X$ in the image of~$M$ containing an edge $e$ of some $\delta_j$, and let $\sigma_X$ be a $2$-cell of $B_j$ containing~$e$.  Let $\sigma, \tau$ be the projections of $\sigma_X,\tau_X$ to the Cayley complex $\Sigma$ of $W_\Gamma$. Let~$\Lambda_\sigma$ be the projection to $\Sigma$ of $\Lambda$,  which intersects $\sigma$ since $D_{i_j+1}$
lies in the same block as $\sigma_X$, and hence projects also to $\sigma$. By Lemma~\ref{lem:app} there is a hypergraph $\Lambda_\tau$ intersecting $\tau$ that is disjoint from or equal to $\Lambda_\sigma$. Let $M'$ be a diagram obtained from $M$ by attaching along $\delta_j$ diagrams in $B_j$ and such that $\gamma$ traverses consecutively its boundary $2$-cells.
The component of the preimage of $\Lambda_\tau$ in $M'$ intersecting $\tau_X$ forms a cycle contradicting the minimality of $M$. This justifies the claim.

Thus $M$ is a tree. Each leaf of that tree is a trivial $\delta_j$, which is a contradiction.
\end{proof}

\bibliography{mybib}{}
\bibliographystyle{plain}

\end{document}